\documentclass[10pt,A4,conference]{IEEEtran}
\usepackage{geometry}
  \usepackage{authblk}
 \usepackage{tikz}  
\usetikzlibrary{calc,intersections}
 \usepackage{verbatim}
\usepackage{textcomp}
\usepackage[T1]{fontenc}
\usepackage{bbm}
\usepackage{multirow}
\usepackage{dsfont}
\usepackage{cite}
\usepackage{epsfig}
\usepackage{float}
\usepackage{enumerate}
\usepackage{balance}
\usepackage{color}
\usepackage{subcaption}
\usepackage{caption}
\usepackage{algorithm}
\usepackage[noend]{algpseudocode}
\algrenewcommand\Return{\State \algorithmicreturn{} } 
\usepackage{hyperref}
\hypersetup{
    colorlinks=true,
    linkcolor=blue,
    filecolor=magenta,      
    urlcolor=cyan,
		citecolor=blue,
}
\usepackage{amssymb}
\usepackage{amsthm}
\usepackage{array}
\usepackage{graphicx}
\usepackage{lettrine}
\usepackage{epsf} 
\usepackage{psfrag}
\usepackage[usenames,dvipsnames]{pstricks}
\usepackage{pst-grad} 
\usepackage{pst-plot} 
\usepackage[space]{grffile} 
\usepackage{etoolbox} 
\usepackage{url} 
\usepackage{algcompatible}
\usepackage{multirow}

\newtheorem{theorem4}{Theorem}
\newtheorem{lemma}[theorem4]{Lemma}

\usepackage[ short ]{optidef}

\usepackage{cuted}
\begin{document}

\title{Ensemble DNN for Age-of-Information Minimization in UAV-assisted Networks }

\author[]{Mouhamed Naby Ndiaye}
\author[]{El Houcine Bergou}
\author[]{Hajar El Hammouti}
\affil[]{College of Computing, Mohammed VI Polytechnic University (UM6P), Benguerir, Morocco, \authorcr
emails: {\{naby.ndiaye,elhoucine.bergou,hajar.elhammouti\}@um6p.ma}}


\maketitle

\begin{abstract}
This paper addresses the problem of Age-of-Information (AoI) in UAV-assisted networks. Our objective is to minimize the expected AoI across devices by optimizing UAVs' stopping locations and device selection probabilities. To tackle this problem, we first derive a closed-form expression of the expected AoI that involves the probabilities of selection of devices. Then, we formulate the problem as a non-convex minimization subject to quality of service constraints. Since the problem is challenging to solve, we propose an Ensemble Deep Neural Network (EDNN) based approach which takes advantage of the dual formulation of the studied problem. Specifically, the Deep Neural Networks (DNNs) in the ensemble are trained in an unsupervised manner using the Lagrangian function of the studied problem. Our experiments show that the proposed EDNN method outperforms traditional DNNs in reducing the expected AoI, achieving a remarkable reduction of $29.5\%$. 

\end{abstract}

\IEEEoverridecommandlockouts
\begin{IEEEkeywords}
Age-of-Information, DNN, Ensemble DNN, Trajectory optimization, UAV-assisted networks, Unsupervised Learning.
\end{IEEEkeywords}
\IEEEpeerreviewmaketitle


\section{Introduction}

Over the past few years, there has been a significant surge in research around the concept of Age of Information (AoI). This interest is driven by various network applications that require timely information to carry out some specific tasks. Examples of such applications include providing real-time traffic to smartphone users and delivering status updates to smart systems~\cite{yates2021age,kadota2021age}. For such applications, the AoI is an important metric as it measures the freshness of the data and evaluates how quickly the data update reaches the destination~\cite{yang2020age,wang2020priority}.

In this paper, we are interested in the scenario where a set of unmanned aerial vehicles (UAVs) is deployed to gather time-sensitive data and send it to a server for analysis and decision-making~\cite{bajracharya20226g,wei2022uav}. To this end, the UAVs should dynamically adjust their trajectories and strategically select the subsets of users from whom data is collected so that the AoI is minimized. Specifically, we answer the question: what is the optimal frequency (or equivalently, the probability to select users) at which UAVs should visit and gather data from devices, and what are the optimal locations of UAVs over time so that the global AoI is minimized?

\subsection{Related work}
Minimizing the AoI in UAV-assisted networks is a daunting task. First, the dynamic movements of UAVs which are often constrained by limited energy resources make the optimization problem a challenging task. Second, the distribution of IoT devices and users across the target area can be uneven, which makes balancing data collection to minimize AoI across all users a complex problem. 

Recently, many works have investigated the AoI minimization in UAV-assisted networks. In~\cite{abd2018average}, the authors minimize the peak of AoI between source-destination pairs. To this end, the authors simultaneously optimize the UAV's flight trajectory and service time for packet transmissions. To solve the problem, they propose an iterative approach where the initial optimization is divided into sub-problems. Each sub-problem is solved analytically and a closed-form expression of the sub-solution is provided. Similarly, 
in~\cite{gao2023aoi}, the problem of the average peak of AoI is divided into two sub-problems. First, a clustering algorithm is proposed to determine the locations of data collection points. Then, the collection points are grouped into clusters, and finally, the flight trajectories of the UAVs are optimized using an ant colony optimization algorithm. In~\cite{ndiaye2022age}, a probabilistic approach is proposed to minimize the probabilities of associations between users and UAVs, and UAVs and the base station. The authors propose a convex reformulation of the problem, which is then solved numerically. The previously cited works propose heuristics to solve the AoI optimization. These methods suffer from several limitations. First, they do not scale well with high-dimensional variables. Additionally, their convergence time is considerably long, and they lack the ability to adapt and generalize to new setups~\cite{arulkumaran2017deep}. 

To overcome these limitations, machine learning (ML) based approaches have been proposed. In~\cite{sun2021aoi}, the authors propose a deep-learning based method to obtain an efficient solution for the flight speed and the trajectory of a single UAV that collects data from IoT devices. A similar approach is proposed in~\cite{liu2021average} where the AoI of ground users is minimized by simultaneously optimizing the trajectory of a UAV, the scheduling of information transmission, and energy harvesting for the ground users. The proposed approach uses a deep reinforcement learning (DRL) to efficiently find optimal solutions. In~\cite{zhu2022uav}, the authors tackle the problem of AoI minimization  by using a transformer network that outputs the optimal visiting order for the ground clusters. The transformer network is combined with a weighted A* algorithm that is used to determine the most suitable hovering point for each cluster. Unlike previous works which considered a single UAV setup, 
 the authors in~\cite{naby2023muti} consider a multi-UAV setup where the AoI is minimized. They introduce  a centralized multi-agent reinforcement learning approach to optimize the UAV trajectories. The proposed scheme relies on a centralized training where information about the environment is shared between UAVs, and a decentralized  execution.

Our paper presents two distinctive differences from existing works. First, instead of considering the association variables as binary, it rather deals with the probability that a UAV visits a given device. This probability can be interpreted as the frequency at which a UAV collects data from a device during its flight. Accordingly, the event of collecting data becomes stochastic,  which justifies the use of the \textit{expected AoI} as a target of our optimization problem. Second, unlike existing works, we propose a novel approach where a collection of DNNs is trained using unsupervised learning. The proposed approach guarantees accurate and robust results.

\subsection{Contribution}
In this paper, we aim to minimize the expected AoI while optimizing the stopping locations of UAVs and the probabilities of device selection. The probabilities of device selection can be interpreted as the frequencies at which UAVs visit devices during the target time. Our contributions can be summarized as follows. 

\begin{itemize}
    \item First, we provide a closed-form expression of the expected AoI of the network which involves the probabilities of selection of devices. Then, we formulate the AoI minimization problem as an optimization with quality of service constraints. 
    \item To address the studied problem, we leverage the framework of EDNNs. The EDNN is based on training a collection of DNNs. Each DNN is trained individually in an unsupervised manner. The training of DNNs relies on the primal-dual formulation of the initial optimization problem.
    \item Our simulation results show that the proposed EDNN approach outperforms traditional DNNs, leading to a reduction of $29.5\%$ of the expected AoI.
\end{itemize}

\subsection{Organization}

The remainder of the paper is organized as follows. First, the system model is described in Section~\ref{Sys}. The mathematical formulation of the problem is given in Section~\ref{Prob}. In Section~\ref{Ensemble}, we describe in details the proposed solution. Next, in Section~\ref{Simu}, we show the performance of our appraoch using simulation experiments. Finally, concluding remarks are provided in Section~\ref{Conc}.


\section{System Model}\label{Sys}
Consider a wireless network where a set $\mathcal{I}$ of $I$ IoT devices periodically generate data updates. The data is transmitted to a server located at the base station (BS). Due to the restricted communication range of IoT devices, a set $\mathcal{U}$ of UAVs is deployed to collect data updates from IoT devices at regular intervals, and then re-transmit the collected data to the BS. During each time interval $t \in \mathcal{T}\triangleq\{0,\dots,T-1\}$, an IoT device $i$ sends its data to a UAV $u$ with probability $p_{i,u}[t]$ using the air-to-ground channel. Our aim is to timely collect the generated data so that its expected age is minimized. 
\subsection{Communication Model}
To model the uplink channel between device $i$ and UAV $u$, we assume a block Rician-fading model, where the channel conditions remain constant over a time interval $t$. As a consequence, the channel response between device $i$ and UAV $u$ at time step $t$ is given by

\begin{equation*}
    h_{i,u}[t] = \sqrt{\frac{\Phi}{\Phi+1}} {\xi}^{\rm LoS}_{i,u}[t] + \sqrt{\frac{1}{\Phi+1}} {\xi_{i,u}^{\rm NLoS}}[t],
\end{equation*}
where $\Phi$ represents the Rician factor, ${\xi}^{LoS}_{iu}[t]$ is the line-of-sight (LoS) component with magnitude $\left|{\xi}^{\rm LoS}_{i,u}[t]\right| = 1$, and ${\xi}^{\rm NLoS}_{i,u}[t]$ is the random non-line-of-sight (NLoS) component following a Rayleigh distribution with mean zero and variance one. 

Let $(x_u[t], y_u[t], H_u)$ be the 3D position of UAV $u$ at time interval $t$, where $H_u$ is the altitude of UAV $u$ that is assumed fixed. Similarly, we denote by $(x_i, y_i, 0)$ the position of device $i$. 
Hence, the distance between device $i$ and UAV $u$ during time interval $t$ is given by $d_{i,u}[t]=\sqrt{(x_{u}[t] - x_{i})^2 + (y_{u}[t] - y_{i})^2 + (H_{u})^2}$. 
We assume that devices use orthogonal frequency division multiple access (OFDMA) to communicate with the UAVs. Hence, the signal-to-noise ratio (SNR) of IoT device $i$ with respect to UAV $u$ at time slot $t$ is given by

\begin{equation*}
    \Gamma_{i,u}[t] = \frac{P_{i}[t] \left|h_{i,u}[t]\right|^{2}}{\sigma^{2}d_{i,u}[t]^{2}},
\end{equation*}
where $P_i[t]$ is the transmit power of device $i$ during time interval $t$, and $\sigma^{2}$ is the variance of an additive white Gaussian noise. Accordingly, the rate of IoT device $i$ with respect to UAV $u$ during time slot $t$ can be expressed as

\begin{equation*}
    {R}_{i,u}[t] = B_{i,u}[t] \log_2\left(1 + \Gamma_{i,u}[t]\right),
\end{equation*}
where $B_{i,u}[t]$ is the allocated bandwidth between device $i$ and UAV $u$ during time slot $t$.

We assume that the generated data from IoT devices is stored in a buffer until it is collected by a UAV for transmission. We also suppose that the size of the devices' buffers is large enough to save all the generated data during the entire time span $T$.

For a successful and efficient data transmission between device $i$ and UAV $u$ during time interval $t$, the data rate $R_{i,u}[t]$ between device-UAV pair should exceed a predefined threshold denoted as $R^{\rm min}$. This threshold $R^{\rm min}$ is carefully chosen to guarantee that data updates can be transmitted almost instantaneously, ensuring rapid and reliable communication between the device and the UAV.


Throughout their flights, UAVs make stops to collect data from subsets of IoT devices. We assume that the data collection time is negligible compared to the overall flight time. We also assume that the UAVs maintain a constant speed $V$ during their flight. Consequently, the total flight time of UAV $u$, denoted as $\zeta_u$, can be expressed as:

\begin{align*}
\zeta_{u}(\!\boldsymbol{x}_u,\boldsymbol{y}_u\!)\!=\!\!
   \!\sum_{t=0}^{T-1}\!\!\frac{\sqrt{\!(\!x_{u}[\!t+1\!]\!-\!x_{u}[t]\!)^2\!+\!(\!y_{u}[\!t+1\!]\!-\!y_{u}[t]\!)^2}}{V}.
\end{align*}

\subsection{Age of Information}

The objective of this work is to optimize the UAVs' 3D locations over time jointly with the probabilities to collect data while maximizing the freshness of the data updates. In particular, our aim is to minimize the expected AoI.
The AoI is defined as the time elapsed between the last update is successfully received by the UAV. Let $\alpha_{i,u}[t]$ be the probabilistic event that UAV $u$ collects data from IoT device $i$ at time interval $t$, and let  $p_{i,u}[t]$ be the probability that $\alpha_{i,u}[t]=1$. Specifically, 

\begin{equation}
\alpha_{i, u}[t]=\left\{\begin{array}{lc}
1, & \text { with probability } p_{i,u}[t]  \\
0, & \text { with probability } 1-p_{i,u}[t].
\end{array}\right.
\label{alphauit}
\end{equation}

We define the AoI of IoT device $i$ with respect to UAV $u$ at time interval $t\geq 1$ using 
a recursive formula as follows

\begin{equation}
    A_{i,u}[t]=\left(A_{i,u}[t-1]+1\right)\left(1-\alpha_{i,u}[t]\right),
    \label{Aiu}
\end{equation}
where $A_{i,u}[0]=0$. Accordingly,  when the data updates of device $i$ are not collected during time interval $t$ (i.e., $A_{i,u}[t]=0$), the AoI is increased by one unit of time. Inversely, when the updates are transmitted, the AoI is reinitialized to zero. In this context, it is judicious to consider the expected AoI with respect to the probabilities of data collection over a number of intervals $T$. The following lemma provides a closed-form expression of the expected AoI.



\begin{lemma}
    The expected AoI $\mathbb{E}(A_{i,u}[t])$ for an IoT device $i$ associated with UAV $u$ at time step $t$ can be expressed as

\begin{equation}\label{AgeClosedForm}
\begin{aligned}
& \mathbb{E}(A_{i,u}[t])\\&= \overline{p}_{i,u}[t]\! \left(\!1 \!+ \!\overline{p}_{i,u}[t-1]\!+\!\!\sum_{k=1}^{t-1}\!\! \left(\prod_{j=1}^{k}\!\overline{p}_{i,u}[j]\overline{p}_{i,u}[j-1]\!\right)\!\!\right),
\end{aligned}
\end{equation}
where $\overline{p}_{i,u}[t] = 1 - p_{i,u}[t]$, and the expectation $\mathbb{E}(.)$ is with respect to the probabilistic event that UAV $u$ collects data from IoT $i$ at time $t$.
\end{lemma}

\begin{proof}
We prove the lemma by induction. 


\textbf{Base Case:} For $t = 1$, we have
\begin{equation*}
\begin{aligned}
& \mathbb{E}(A_{i,u}[t])\!= \overline{p}_{i,u}[1]\! \left(\!1 \!+ \!\overline{p}_{i,u}[0]\!\right)=\overline{p}_{i,u}[1],
\end{aligned}
\end{equation*}
which matches the derived expression.

\textbf{Inductive Step:} Let us assume that the lemma holds for $t = n$, $1<n<T-1$ i.e.,
\begin{equation}\label{ASSU}
\begin{aligned}
& \mathbb{E}(\!A_{i,u}[n]\!)\!=\!\overline{p}_{i,u}[n]\! \left(\!1\! +\! \overline{p}_{i,u}[n-1] \!+\!\sum_{k=1}^{n-2} \left(\!\prod_{j=1}^{k}\!\overline{p}_{i,u}[j]\overline{p}_{i,u}[j-1]\!\!\right)\right).
\end{aligned}
\end{equation}

In the following, we prove that it holds for $t = n+1$.

\begin{equation}
\begin{aligned}
   \mathbb{E}(A_{i,u}[n+1]) &= \mathbb{E}((A_{i,u}[n]+1)(1-\alpha_{i,u}[n+1])\\
   &=(1-p_{i,u}[n+1])\mathbb{E}(A_{i,u}[n])+1-p_{i,u}[n+1].
\end{aligned}
\end{equation}
Using our assumption in equation~(\ref{ASSU}), we replace $\mathbb{E}(A_{i,u}[n])$ by its expression and obtain

\begin{equation*}
\begin{aligned}
   \mathbb{E}(A_{i,u}[n+1]) 
   &=\overline{p}_{i,u}[n+1]\overline{p}_{i,u}[n] (1\! + \overline{p}_{i,u}[n-1] \\&+\!\sum_{k=1}^{n} \left(\!\prod_{j=1}^{k}\!\overline{p}_{i,u}[j]\overline{p}_{i,u}[j-1]\!\!\right))+\overline{p}_{i,u}[n+1]).
\end{aligned}
\end{equation*}

Finally, by arranging the expression above, we obtain

\begin{equation*}
\begin{aligned}
& \mathbb{E}(A_{i,u}[n+1])= \overline{p}_{i,u}[n+1] \left(1 + \overline{p}_{i,u}[n] \right.\\
&+\sum_{k=1}^{n} \left(\prod_{j=1}^{k}\overline{p}_{i,u}[j]\overline{p}_{i,u}[j-1]\right)).
\end{aligned}
\end{equation*}

Therefore, the lemma holds for $t = n+1$. By induction, the lemma is proven for all $t \geq 1$.
\end{proof}
From equation~(\ref{AgeClosedForm}), we can observe that when $p_{i,u}[t]=1$ (or equivalently $\overline{p}_{i,u}[t]=0$) for all $t\in \mathcal{T}$, i.e., the data is collected from user $i$ by UAV $u$ for all time intervals, the corresponding expected AoI becomes zero. Conversely, when $p_{i,u}[t]=0$, i.e., no data has been collected over the considered time, the expected AoI related to user $i$ and UAV $u$ reaches its maximum value which is equal to $T$. 
\section{Problem Formulation}
\label{Prob}
The objective of this work is to minimize the expected AoI across devices during a number of time intervals $T$. The optimization problem involves finding the optimal probabilities $\boldsymbol{p}$ of selecting devices to collect data updates and the stopping points $(\boldsymbol{x},\boldsymbol{y}$) of UAVs over time, while considering various constraints. Accordingly, our problem is formulated as follows

 
\begin{mini!}
{\boldsymbol{p},\boldsymbol{x},\boldsymbol{y}} {\!\!\sum\limits_{\substack{(t,u,i) \in \\ \mathcal{T}\times \mathcal{U}\times \mathcal{I}}}\!\!\!\!\!\!\overline{p}_{i,u}[t]\!\! \left(\!\!1\!\!+\! \overline{p}_{i,u}[t-1] \!+\!\!\sum_{k=1}^{t-1}\!\! \left(\!\prod_{j=1}^{k}\!\overline{p}_{i,u}[j]\overline{p}_{i,u}[j-1]\!\!\!\right)\!\!\!\right)
\label{objective}}
{\label{GeneralOptimizati}}{}
\addConstraint{R_{i,u}[t]\geq p_{i,u}[t]R^{\rm min},\;  \forall (t,u,i) \in \mathcal{T}\!\!\!\times \mathcal{U}\!\!\times \mathcal{I} \label{Rate}}{}{}
\addConstraint{ 
\sum \limits_{u\in \mathcal{U}}p_{i,u}[t]\leq 1, \; \forall (t,i) \in \mathcal{T}\!\!\times \mathcal{I}
\label{Association2}}
{}{}
\addConstraint{ 
\sum \limits_{i\in \mathcal{I}}p_{iu}[t]\leq N_u, \;  \forall (t,u) \in \mathcal{T}\!\!\!\times \mathcal{U}
\label{CapacityUAV}}
{}{}
\addConstraint{ 
\zeta_{u}(\!\boldsymbol{x}_u,\boldsymbol{y}_u\!) \leq T, \;  \forall u \in \mathcal{U}
\label{TimeUAV}}
{}{}
\addConstraint{ 
0\leq x_u[t]\leq x^{\rm max}, \;  \forall (t,u) \in \mathcal{T}\!\!\!\times \mathcal{U}
\label{xposition}}
{}{}
\addConstraint{ 
0\leq y_u[t]\leq y^{\rm max}, \;  \forall (t,u) \in \mathcal{T}\!\!\!\times \mathcal{U}
\label{yposition}}
{}{}
\addConstraint{ 
0\leq  p_{i,u}[t]\leq 1, \;  \forall i \in \mathcal{I}, \forall (t,u) \in \mathcal{T}\!\!\!\times \mathcal{U}
\label{alpha}}
{}{}
\end{mini!}

 Constraint~(\ref{Rate}) ensures that the expected rate between each UAV and its served IoT device is above a predefined threshold $R^{\min}$. Constraint (\ref{Association2}) guarantees that a device can transmit to at most one UAV at a time, on average. Similarly, constraint (\ref{CapacityUAV}) ensures that the expected number of served devices by UAV $u$ does not exceed its maximum capacity $N_u$. Constraint~(\ref{TimeUAV}) guarantees that each UAV $u$ adheres to a maximum flight time, denoted as $\zeta_{\text{max}}^u$, which aligns with its energy budget. Constraints (\ref{xposition}) and (\ref{yposition}) limit UAVs' movements to a specific area. Constraint (\ref{alpha}) bounds the probabilities of device selection between $0$ and $1$.

 Solving the expected AoI minimization is challenging due to the non-convexity of both the objective function and constraints~(\ref{Rate}) and~(\ref{TimeUAV}). To address this problem, we leverage the power of EDNNs. EDNNs take advantage of the impressive ability of DNN to approximate highly complex functions. Specifically, EDNN is a collection of DNNs trained with different initial weights and training data. Each DNN model in the ensemble is individually trained and stored. During the test, the DNNs' results are combined using an aggregation rule (e.g., averaging). 
 
 In the next section, we first explain how a single DNN model can efficiently solve the expected AoI minimization, then, we describe how the EDNN solution is leveraged to provide accurate results.



\section{Ensemble Deep Neural Networks based Approach  }\label{Ensemble}

To address the constrained AoI problem, an alternative approach is to solve its primal-dual formulation. In fact, while the optimal solution of the dual problem may not necessarily be the optimal solution for the original AoI minimization (due to the non-convexity of the problem), it can still offer an efficient local optimum. Specifically, the Lagrangian function for the problem under study is defined in (\ref{LossFunction}).

\begin{figure*}
    \begin{equation} 
\begin{aligned}
L(\boldsymbol{p},\boldsymbol{x},\boldsymbol{y},\boldsymbol{\mu})\!=&\sum \limits_{\substack{(i, u,t)\in\\\mathcal{I}\times \mathcal{U}\times \mathcal{T}}}\!\!\!\mathbb{E}(A_{i,u}[t])\!\!+\!\!\left(\!\sum \limits_{\substack{(i, u,t)\in\\\mathcal{I}\times \mathcal{U}\times \mathcal{T}}} \!\!\!\mu_{i,u,t}^1 C_{i,u,t}^1\!\right)\! +\!\!\left(\!\sum \limits_{\substack{(i, t)\in\\\mathcal{I}\times \mathcal{T}}}\!\! \mu_{i,t}^2 C_{ i,t}^2\!\right)
 \left.+\left(\!\sum \limits_{\substack{( u,t)\in\\\mathcal{U}\times \mathcal{T}}} \mu_{u,t}^3 C_{ u,t}^3\!\right)\!
 \right.+\left.\!\left(\!\sum \limits_{u\in \mathcal{U}} \mu_{u}^4 C_{ u}^4\!\right)\!+\!\left(\sum \limits_{\substack{( u,t)\in\\ \mathcal{U}\times \mathcal{T}}} \mu^5_{u,t} C_{u,t}^5\!\right)\right.
 \\&\left.\!+\!\left(\!\sum \limits_{\substack{(i, u,t)\in\\\mathcal{I}\times \mathcal{U}\times \mathcal{T}}} \mu_{i,u,t}^6 C_{ i,u,t}^6\!\right)\right.
 \!+\!\left(\!\sum \limits_{\substack{(i,u,t)\in\\\mathcal{I}\times \mathcal{U}\times \mathcal{T}}} \mu_{i,u,t}^7C_{ i,u,t}^7\!\right).
 \end{aligned}
 \label{LossFunction}
\end{equation}
\end{figure*}

$(C_{.}^j)_{j=1}^7$ in~ (\ref{LossFunction}) captures the constraints of the problem, which are expressed as follows $C_{i,u,t}^1=\operatorname{ReLU}\left(R^{\rm min}-p_{i,u}[t]R_{i,u}[t]\right)$, $C_{i, t}^2=\operatorname{ReLU}(\sum \limits_{u\in \mathcal{U}}p_{i,u}[t]- 1)$, $C_{ u,t}^3=\operatorname{ReLU}(\sum \limits_{i \in \mathcal{U}}p_{i,u}[t]-N_u)$, $C_{ u}^4=\operatorname{ReLU}\left(\zeta(\boldsymbol{x}_u,\boldsymbol{y}_u)-\zeta^{\rm max}\right)$,
$C_{ u,t}^5=\operatorname{ReLU}\left(x_{u}[t]-x^{\rm max}\right)$,
$C_{ u,t}^6=\operatorname{ReLU}\left(y_{u}[t]-y^{\rm max}\right)$,
$C_{i, u,t}^7=\operatorname{ReLU}\left(p_{iu}[t]-1\right)$, where $\operatorname{ReLU}(x)=\max(0,x)$, is the rectified linear function, and $(\boldsymbol{\mu}^j_{.})_{j=1}^7$ are the non-negative Lagrange multipliers. Accordingly, an alternative formulation to solve the expected AoI minimization problem is given by

\begin{equation}\label{lagr}
    \max \limits_{\{\boldsymbol{\mu}^j\}} \min \limits_{\boldsymbol{x},\boldsymbol{y},\boldsymbol{p}} L(\boldsymbol{p},\boldsymbol{x},\boldsymbol{y},\boldsymbol{\mu}).
\end{equation}

\begin{algorithm}[t]
\caption{EDDN for AoI Minimization}\label{alg:3}
\begin{algorithmic}[1]
\Statex{\textbf{Training phase :}}
\State{ Let $N$ be the number of models in the EDNN.}
\State{
Generate the training data and split it into $N$ datasets}
\FOR{$i=1,\dots,N$}
    \State{Initialize the weights of the $i$th DNN model  $\boldsymbol{w}$ and the Lagrange multipliers $(\boldsymbol{\mu}^j_{.})_{j=1}^7$}
    \FOR{Iterations of stochastic gradients}
    \State{Select a mini-batch of data and use gradient descent  to update the weights of the DNN as follows $\boldsymbol{w}^{t+1}=\boldsymbol{w}^{t}-\eta \frac{\partial \hat{L}(\boldsymbol{w},\boldsymbol{\mu})}{\partial w}$} in order to get $\boldsymbol{x}^{t},\boldsymbol{y}^{t},\boldsymbol{p}^{t}$ .
    \State{ Update the Lagrangian multiplier as follow
    ${(\boldsymbol{\mu}^j)}^{t+1}={(\boldsymbol{\mu}^j)}^t+\beta \frac{\partial \hat{L}(\boldsymbol{w},\boldsymbol{\mu})}{\partial\mu^j}={(\boldsymbol{\mu}^j)}^t+\beta \left(\hat{C}^j\right)$ $\forall j \in {1,\dots,7}$.
    }
\ENDFOR
\ENDFOR
\Statex{\textbf{Testing phase}}
\State{Generate the testing data.}
\FOR{$i$ in $N$}
    \State{Output $\boldsymbol{x}^i$, $\boldsymbol{y}^i$, and $\boldsymbol{p}^i$ (the output vectors of DNN $i$).}
    \State{Compute the expected AoI denoted as $AoI_i$ using (\ref{AgeClosedForm}).}
\ENDFOR
\State{Compute the final output vectors from all DNN models
   $\boldsymbol{x} = \frac{\sum_{i=1}^{N} AoI_i \cdot \boldsymbol{x}^i }{\sum_{i=1}^{N} AoI_i}$,
   $ \boldsymbol{y} = \frac{\sum_{i=1}^{N} AoI_i \cdot \boldsymbol{y}^i}{\sum_{i=1}^{N} AoI_i} $,
    $\boldsymbol{p} = \frac{\sum_{i=1}^{N} AoI_i \cdot \boldsymbol{p}^i }{\sum_{i=1}^{N} AoI_i}$.}
\end{algorithmic}
\end{algorithm}

To solve problem~(\ref{lagr}), we leverage the ability of DNN to approximate complex functions. To this end, the data collection probabilities and the scheduling of UAVs locations are modeled as an output of a DNN. Specifically, \begin{equation}
(\boldsymbol{x},\boldsymbol{y},\boldsymbol{p})\triangleq f(\boldsymbol{w};\theta),
\end{equation}
where $\boldsymbol{w}$ is the DNN's vector of weights, $\theta$ is the input vector composed of environment parameters (e.g., the transmit powers, bandwidth, channel gains, etc) and $f(.)$ is the DNN model. Hence, to find the optimal data collection probabilities and effectively schedule their locations over time, we adopt an unsupervised learning approach. This approach differs from traditional supervised learning, where the DNN's training depends on a numerical algorithm to solve the optimization problem. Instead, we utilize the Lagrangian of the optimization problem as a cost function to train the DNN using an unsupervised learning. Moreover, to optimize the Lagrangian multipliers, we employ a gradient ascent optimization. 

Accordingly, at the $t^{\rm th}$ iteration, the variables of problem~(\ref{lagr}) are optimized using stochastic gradient descent and gradient ascent as follows  
\begin{equation}
\boldsymbol{w}^{t+1}=\boldsymbol{w}^{t}-\eta \frac{\partial \hat{L}(\boldsymbol{w},\boldsymbol{\mu})}{\partial w},
\end{equation}
and for all $j \in \{1,\dots,7\}$
\begin{equation}{(\boldsymbol{\mu}^j)}^{t+1}={(\boldsymbol{\mu}^j)}^t+\beta \frac{\partial \hat{L}(\boldsymbol{w},\boldsymbol{\mu})}{\partial\mu^j}={(\boldsymbol{\mu}^j)}^t+\beta \left(\hat{C}^j\right),\end{equation}
where $\boldsymbol{w}^t$ and ${(\boldsymbol{\mu}^j)}^t$ are the vectors of weights and Lagrangian multipliers at the $t^{\rm th}$ iteration, respectively. $\hat{L}(\boldsymbol{w},\boldsymbol{\mu})$ is the expected value of the Lagrangian function applied to a batch of input data. Similarly, $\hat{C}^j$ is the expected value of the $j^{\rm th}$ constraint applied to a batch of input data. Finally, $\eta$ and $\beta$ are the learning rates of stochastic gradient descent and ascent, respectively. 

 \begin{figure*}[t]
\centering
\minipage{0.29\textwidth}
  \includegraphics[width=1\linewidth]{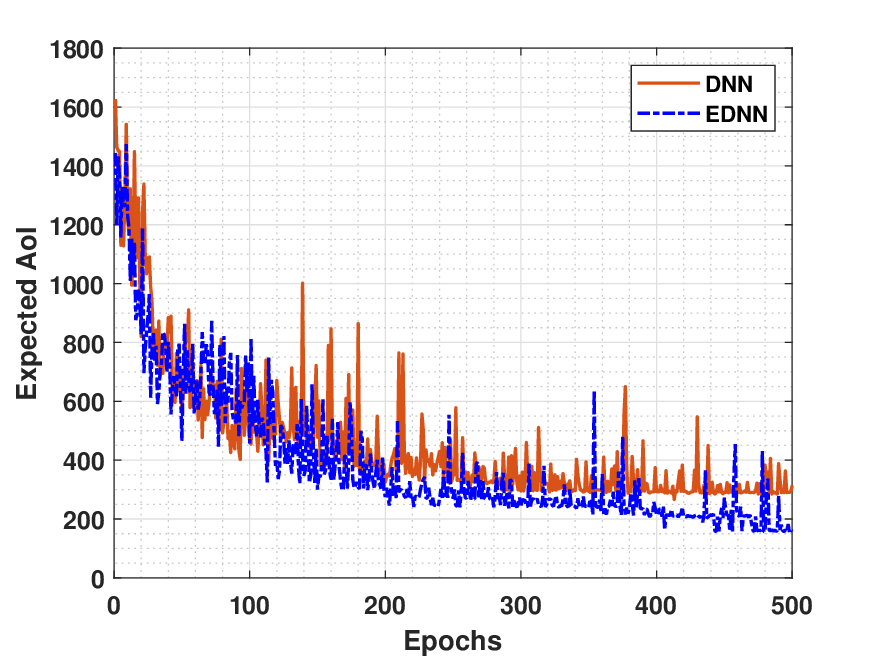}
    \caption{Expected AoI vs epochs}
    \label{fig:AoIvsepoch}
\endminipage\hfill
\minipage{0.29\textwidth}
  \includegraphics[width=1\linewidth]{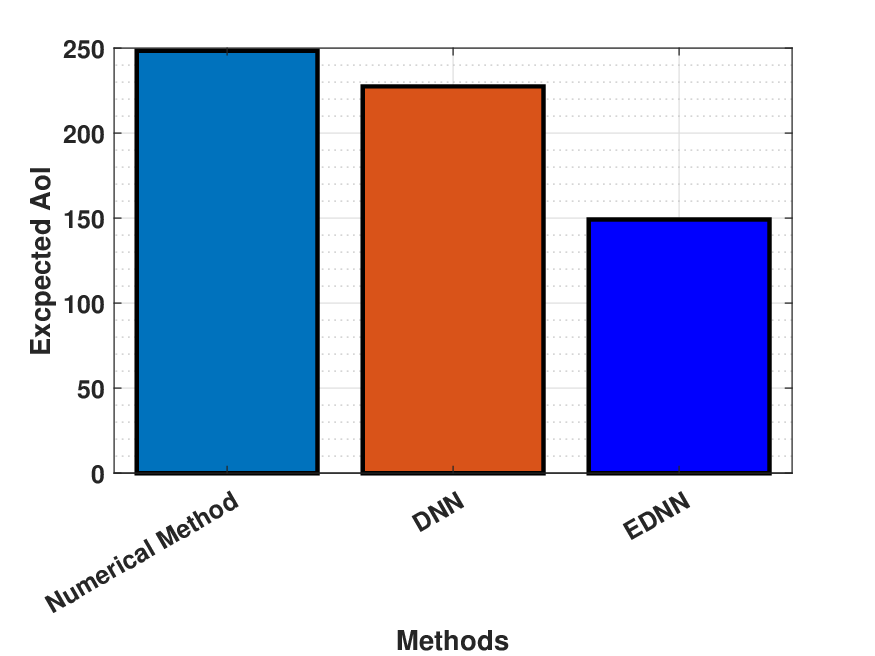}
    \caption{Expected AoI during testing phase}
    \label{fig:AoItestphase}
\endminipage \hfill
\minipage{0.29\textwidth}
  \includegraphics[width=1\linewidth]{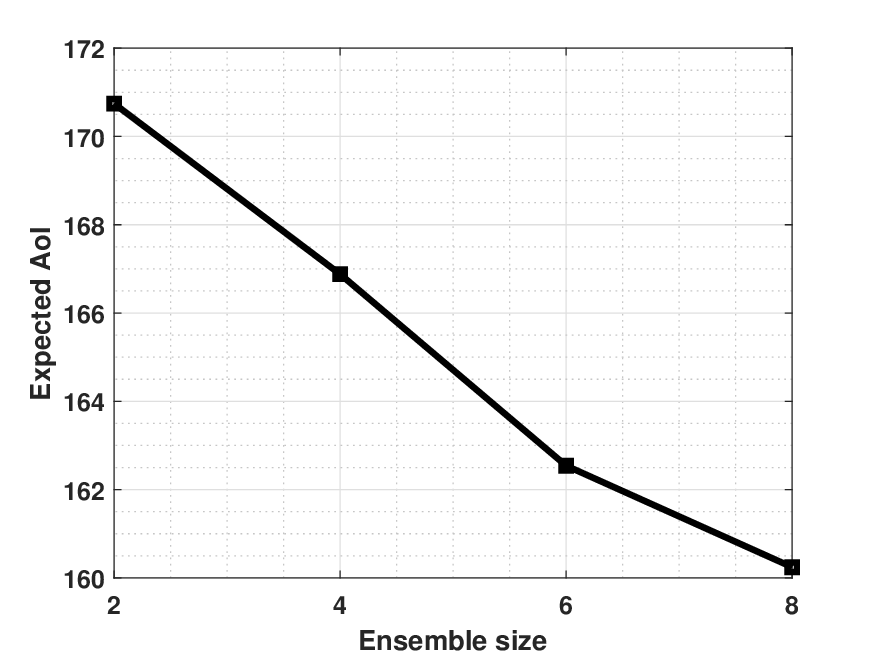}
    \caption{Average expected AoI vs ensemble size}
    \label{fig:aoivsednnsize}
\endminipage
\end{figure*}

The DNN is trained with the aim to output the optimal UAV positions $\boldsymbol{x}$, $\boldsymbol{y}$, and data collection probabilities $\boldsymbol{p}$. To ensure that each UAV's position remains within the target area, a $\operatorname{ReLU}$ function is applied to outputs related to UAVs' 2D positions at the output layer. Similarly, a $\operatorname{Sigmoid}$ function is applied to the outputs related to data collection probabilities. These activation functions guarantee that the DNN's outputs are bounded within the specified intervals.

To enhance the generalization performance and ensure the robustness of the proposed DNN, we leverage the framework of EDNNs.
Compared to DNNs, EDNNs combine multiple DNN models into an ensemble, which leads to enhanced accuracy and robustness~\cite{ganaie2022ensemble}. In the context of AoI minimization, the EDNN will improve the generalization ability of the model for unseen scenarios and handle the uncertainty of the wireless environment. In the following, we describe how the EDNN is efficiently trained and tested. The description of the proposed approach is provided in algorithm~\ref{alg:3}.

\begin{enumerate}
    \item \textbf{Training EDNN}: We implement an EDNN structure in which the DNN models within the ensemble share the same architecture. However, each DNN is initialized and trained using different initial weights and training sets.
In fact, to achieve an efficient training of EDNNs and avoid overfitting,
 it is important to ensure a minimal overlap in the datasets used to train each DNN within the ensemble.
For the studied AoI problem, the input data is composed of the channel gains, the bandwidth allocations, the transmit power, and the locations of IoT devices. The data is generated randomly and is equally divided between the DNNs within the ensemble. Then, each DNN is initialized randomly. At each iteration of the training, a mini-batch is randomly selected to perform gradient descent and ascents updates. 

     \item \textbf{Testing EDNN}:
During the testing phase, the test set is drawn from the same distribution as the training data. 
Each trained DNN is provided with the test data and produces the UAVs' scheduling and data collection probabilities. The final output of the EDNN is computed by taking a weighted average of all the output vectors, where the weights are proportional to the AoI provided by each DNN within the ensemble.
\end{enumerate}

It is important to note that due to the computational complexity of the training and the relatively small gain 
 in performance that comes from adding multiple DNNs, the number of models in EDNN is kept small (generally up to $10$).


%
\section{Simulation Results}\label{Simu}

To evaluate the performance of the proposed approach, we consider an area of $1000m\times1000m$, where a number of $30$ IoT devices are randomly scattered. We also suppose that $3$ UAVs are deployed to collect and keep the data as fresh as possible. The UAVs hover at altitudes between $80m$ and $100m$.  Moreover, the devices are assigned a fixed bandwidth, randomly picked between $[1.5,2]$ GHz and a constant power between $[0,1]$ mWatt. To satisfy the quality of service constraint, the minimum rate is set to $150$ Kbit/s.

The parameters of our simulation setup are summarized in Table \ref{tab:my-table-parameter}.

\begin{table}[ht]
\centering
\begin{tabular}{|l|l||l|l|}
\hline
\textbf{Parameter} & \textbf{Value} & \textbf{Parameter} & \textbf{Value}   \\ \hline
$I$       & $30$      &$U$       & $3$  \\ \hline
$x_{max}$       & $1000m$ & $y_{max}$       & $1000m$  \\ \hline
$H_u$       & $[80,100]m$   & $R_{min}$       & $150$ Kbit/s  \\ \hline
$N_u$       & $8$   &$T$       & $40$   \\ \hline
$B_{i,u}$       & $[1.5,2]$ GHz   &$\sigma^{2}$       & $-120dBm$  \\ \hline
$P_{i}$       & $[0,1]$ mW   &$T$       &$40$ \\ \hline
\end{tabular}
\caption{Experiment setup}
\label{tab:my-table-parameter}
\end{table}

The mini-batch size and the gradient descent learning rate are taken as 50 and $0.001$. At each iteration, the number of epochs is 150. Moreover, the step size $\beta$ of updating the penalty parameters is set to $0.1$. For each single DNN, the number of neurons from the input layer to the output layer is given as $\{600,1200,2400,4800\}$. Finally, the ensemble size is set to $8$. 


In Fig.~ \ref{fig:AoIvsepoch}, we observe a consistent reduction in the expected AoI for both DNN and EDNN during the training phase. Moreover, it can be seen through the figure that by the end of the training the EDNN achieves a substantial decrease in the average AoI. These results are confirmed through the testing phase as illustrated by Fig.~\ref{fig:AoItestphase}. Specifically, Fig.~\ref{fig:AoItestphase} plots the expected AoI in the test and compare it with DNN and  a numerical method (based on the interior- point algorithm). As it can be seen through the figure, the EDNN approach outperforms DNN and the numerical method  as it achieves a reduction of approximately $29.5\%$ compared to DNN, and a reduction of approximately $35.5\%$ compared to the numerical method.

In Figure \ref{fig:aoivsednnsize}, we investigate the impact of the ensemble size in EDNN on the achieved expected AoI during the test. It can be seen from the figure that as the ensemble size increases, the expected AoI is further minimized, which indicates the potential for even better performance with larger ensemble sizes.






\section{Conclusion}\label{Conc}

In this paper, we studied the problem of AoI minimization in UAV-assisted networks. Specifically, we proposed an EDNN based approach to efficiently schedule the 2D positions over time and optimize the probabilities of selection. The EDNN is trained using an unsupervised learning which relies on the minimization of the Lagrangian function of the studied problem. Our simulation results show that the proposed approach outperforms traditional DNN in minimizing the AoI.

\section*{Acknowledgment}
This document has been produced with the financial assistance of the European Union (Grant no. DCI-PANAF/2020/420-028), through the African Research Initiative for Scientific Excellence (ARISE), pilot programme. ARISE is implemented by the African Academy of Sciences with support from the European Commission and the African Union Commission.






\balance
\bibliographystyle{IEEEtran}
\bibliography{IEEEabrv,biblio_traps_dynamics}

\end{document}